\documentclass[12pt]{amsart}
\usepackage{amsmath,amssymb,amsfonts}
\usepackage{enumerate}

\input xy
\xyoption{all}

\theoremstyle{plain}
\newtheorem{theo}{Theorem}
\newtheorem{thm}{Theorem}[section]

\newtheorem{lem}[thm]{Lemma}

\newtheorem{coro}[thm]{Corollary}

\newtheorem{prop}[thm]{Proposition}

\theoremstyle{definition} \theoremstyle{definition}
\newtheorem{defn}[thm]{Definition}

\newtheorem{rem}[thm]{Remark}           

\theoremstyle{remark}

\newcommand{\A}{\mathbb{A}}
\newcommand{\Q}{\mathbb{Q}}
\newcommand{\bQ}{\mathbb{Q}}

\newcommand{\Z}{\mathbb{Z}}

\newcommand{\R}{\mathbb{R}}

\newcommand{\C}{\mathbb{C}}
\newcommand{\N}{\mathbb{N}}

\newcommand{\Ker}{\mathrm{Ker}}

\def\PGL{{\rm PGL}}

\def\Gal{{\rm Gal}}

\def\Pic{{\rm Pic}}

\def\vol{{\rm vol}}
\def\rH{{\mathrm H}}
\def\rZ{{\mathrm Z}}
\def\ra{{\rightarrow}}
\def\ux{{\underline{x}}}

\def\sH{{\mathsf H}}
\def\sN{{\mathsf N}}

\begin{document}

\title[Multiple mixing]{Multiple mixing for adele groups and rational points}

\author{Alexander Gorodnik}
\address{School of Mathematics\\
University of Bristol\\
Bristol BS8 1TW, U.K.
}
\email{a.gorodnik@bristol.ac.uk}

\author{Ramin Takloo-Bighash}
\address{University of Illinois at Chicago \\
851 S. Morgan str. (M/C 249) \\
Chicago, IL 60607, U.S.A.}
\email{rtakloo@math.uic.edu}

\author{Yuri Tschinkel}
\address{Courant Institute of Mathematical Sciences, N.Y.U. \\
 251 Mercer str. \\
 New York, NY 10012, U.S.A.}
\email{tschinkel@cims.nyu.edu}        


\begin{abstract}
We prove an asymptotic formula for the number of rational points of bounded height on 
projective equivariant compactifications of $H\backslash G$, where $H$ is a connected simple 
algebraic group embedded diagonally into $G:=H^n$. 
\end{abstract}

\maketitle

\section*{Introduction}
\label{sect:intro}

Let $X\subset \mathbf P^n$ be a smooth projective variety over a number field $F$. Fix a height function 
\begin{equation}
\label{eqn:choice}
\sH\colon \mathbf P^n(F)\ra \R_{>0}
\end{equation}
and consider the counting function 
$$
\sN(X,T):=\{ x\in X(F) \,|\, \sH(x)\le T\}.
$$
 Manin's conjecture \cite{FMT} and its refinements by Batyrev--Manin \cite{BM}, Peyre \cite{Peyre}, 
and Batyrev--Tschinkel \cite{BT} predict precise asymptotic formulas for 
$\sN(X^{\circ},T)$ as $T\ra \infty$, 
where $X^{\circ}\subset X$ is an appropriate Zariski open subset of an algebraic variety with 
sufficiently positive anticanonical class. These formulas involve geometric invariants of $X$:
\begin{itemize}
\item the Picard group $\Pic(X)$ of $X$; 
\item the anticanonical class $-K_X\in \Pic(X)$;
\item the cone of 
pseudo-effective divisors $\Lambda_{\rm eff}(X)_{\R}\subset \Pic(X)_{\R}$,
\end{itemize}
and they depend on an {\em adelic metrization} $\mathcal L = (L, \|\cdot\|_v)$ of the polarization $L$
giving rise to the embedding $X\subset \mathbf P^n$, i.e., on a choice of the height function in \eqref{eqn:choice}.  
Given these, one introduces the invariants:
$$
a(L), b(L), \, \text{ and }\, c(\mathcal L)
$$
so that the number of $F$-rational points on $X^{\circ}$ of $\mathcal L$-height bounded by $T$ is, 
conjecturally, given by
\begin{equation}
\label{eqn:main}
\sN(X^{\circ},\mathcal L,T)=\frac{c(\mathcal L)}{a(L)(b(L)-1)!} T^{a(L)}\log(T)^{b(L)-1}(1+o(1)), \quad T\ra \infty,
\end{equation}
see, e.g., \cite{BT} for precise definitions of the constants. 

\

These conjectures have stimulated intense research;  
see \cite{T}, \cite{Oh}, \cite{browning}, \cite{CL} for surveys of the current state of this subject.
Of particular importance are equivariant compactifications of algebraic groups and their homogeneous spaces.
In all equivariant cases considered previously, it was essential that $X$ admits an action, with a dense orbit, 
of a solvable algebraic group. For example, the paper \cite{STT} proves Manin's conjecture 
for equivariant compactifications of the symmetric space $G\backslash (G\times G)$, a spherical variety. 
In this paper, we establish Manin's conjecture for a new class of varieties, 
which includes {\em nonspherical} varieties. 

\begin{theo}
\label{thm:main}
Let $H$ be a connected simple algebraic group defined over a number field $F$, 
$G:=H^r$ its $r$-fold product.
Let $X$ be a smooth projective $G$-equivariant compactification of $H\backslash G$, 
where $H$ acts on the left diagonally. 
Then $X$ satisfies Manin's conjecture, i.e., \eqref{eqn:main} holds for $L=-K_X$ and 
$$
X^{\circ} = H\backslash G \subset X, 
$$ 
and some non-zero constant $c(\mathcal{L})$. 
\end{theo}

This generalizes the case $r=2$ treated in \cite{STT} and \cite{GMO} to arbitrary $r$. The proof presented here also 
works, with minor modifications, for semi-simple groups $H$. 
Compactifications of  the homogeneous space $H\backslash H^r$ have played an important role in work of L. Lafforgue on the Langlands' conjecture
over function fields of curves over finite fields (see, e.g., Chapter 3 in \cite{lafforgue}). 
The geometry of these compactifications is surprisingly rich.

\

Our proof combines ergodic-theoretic methods developed in \cite{GO} with geometric integration techniques developed in 
\cite{chambert-t02} and \cite{CLT-igusa}; in particular, it uses neither the theory of height zeta functions
nor spectral theory on adelic spaces. On the other hand, it does not allow to establish effective error terms
as in the $r=2$ case in \cite{STT}.  

\

{\bf Organization of the paper.}
In Sections \ref{sect:geometry} and \ref{sect:heights} we 
discuss geometric and analytic background and, in particular, establish 
meromorphic continuation of Igusa-type integrals 
(Theorem \ref{thm:hecke}) that implies an asymptotic formula for
volumes of height balls.  In Section \ref{sect:intermediate},
we give a classification of intermediate subgroups $M$ with
$H\subset M\subset H^r$. This result is used in Section \ref{sect:mixing}
where we establish the multiple mixing property for the adelic spaces
using measure-rigidity techniques.
Finally, our main result is deduced from multiple mixing in Section \ref{sect:count}.

\

\noindent
{\bf Acknowledgments.} 
We are grateful to Antoine Chambert-Loir, Amos Nevo, and 
Brendan Hassett for useful comments and suggestions.  
The first author was supported by EPSRC, ERC, and RCUK.
The second authors was partially supported by NSF grant DMS-0701753 and by the NSA grant 081031. 
The third author was partially supported by NSF grants DMS-0739380 and 0901777.

\section{Geometric background}
\label{sect:geometry}

Let $F$ be an algebraically closed field of characteristic zero,
$G$ a connected semi-simple algebraic group defined over $F$ and $H\subset G$ a connected closed subgroup.
Let $X$ be a projective equivariant compactification of $X^{\circ}:=H\backslash G$. 
Using $G$-equivariant resolution of singularities, we may assume that $X$ is smooth and that 
the boundary
$$
\cup_{\alpha\in \mathcal A} D_{\alpha} = X\backslash X^{\circ}
$$
is a divisor with normal crossings. 
If $H$ is a parabolic subgroup, then there is no boundary, i.e., $\mathcal A$ is empty,
and $H\backslash G$ is a generalized flag variety. Distribution of rational points of bounded
height on flag varieties was studied in \cite{FMT}. 

Throughout, we will assume that 
\begin{itemize}
\item $\mathcal A$ is not empty,
\item $X^{\circ}$ is affine (this holds, e.g., when $H$ is reductive), 
\item the groups of algebraic characters of $G$ and $H$ are trivial.
\end{itemize}  


Recall that a 1-parameter subgroup of $G$ is a homomorphism $\xi:\mathbb G_m\ra G$. 

\begin{lem}
\label{lemm:1-par}
Let $X$ be a smooth projective $G$-equivariant compactification of $H\backslash G$. 
Then for every boundary divisor $D_{\alpha}$, 
there exists a 1-parameter subgroup $\xi_{\alpha} : \mathbb G_m\ra G$ such that the generic point of 
$D_{\alpha}$ is in the limit of $\xi_{\alpha}$. 
\end{lem}

\begin{proof}
See, e.g., \cite[Proposition 4.2]{birkes}.
\end{proof}

We will identify line bundles and divisors with their classes. 

\begin{prop}
\label{prop:picard}
Let $G$ be a connected reductive group, $H\subset G$ a closed connected reductive subgroup, and 
$X$ a smooth projective $G$-equivariant compactification of $X^\circ = H\backslash G$.
Assume that $G$ and $H$ have no nontrivial algebraic characters.
Then 
\begin{enumerate}
\item the classes of irreducible boundary components $D_{\alpha}$ span the Picard group $\Pic(X)_{\Q}$ and 
the pseudo-effective cone $\Lambda_{\rm eff}(X)\subset \Pic(X)_{\R}$;
\item the class of the anticanonical line bundle 
is given by
$$
-K_X= \sum_{\alpha\in \mathcal A} \kappa_{\alpha}D_{\alpha}, 
$$
where all $\kappa_{\alpha} \ge 1$. 
\end{enumerate}
\end{prop}

\begin{proof}
Fix a polarization $L$ of $X$ and let $X\subset \mathbf P^n$ be
the corresponding projective embedding.  
After taking a suitable multiple, we 
may assume that $L$ is $G$-linearized, i.e., the action of $G$ on $X$ extends 
to an action on the ambient $\mathbf P^n$ (by \cite[Corollary 1.6]{GIT}). 
Let $D$ be an effective divisor such that 
the generic point of $D$ is in $H\backslash G$. There exists 
a 1-parameter subgroup moving the generic point of $D$. After specializing,
$D$ breaks and at least one of the irreducible components of the limit is supported in the boundary. 
We can now apply induction on the $L$-degree to conclude that $D$ is equivalent to an effective divisor
with support in the boundary. 

On the other hand, the only invertible functions on $H\backslash  G$ are constants, by assumption. 
It follows that there are no relations between classes of the boundary components. 

For the second claim, see, e.g.,  \cite[Section 6]{balanced}. 
\end{proof}

Let $L$ be a big line bundle on X. We define 
\begin{align*}
&a(L) := \inf \{ t \in \bQ : t[L] + [K_X] \in \Lambda_{\textnormal{eff}}(X)\}, \\
&b(L) := \textnormal{ the maximal codimension of the face containing $a(L)L + K_X$}.
\end{align*}
By Proposition~\ref{prop:picard}, we have
$$
L=\sum_{\alpha\in \mathcal A} \lambda_{\alpha} D_{\alpha}, \quad \lambda_{\alpha} \in \Q_{>0},
$$
so that the corresponding invariants are given by
\begin{equation}
\label{eqn:a}
a(L) := \max_{\alpha} \frac{\kappa_{\alpha}}{\lambda_{\alpha}}
\end{equation} 
and  
\begin{equation}
\label{eqn:b}
b(L) :=\#\{ \alpha\in \mathcal A \, |\, a(L)=\frac{\kappa_{\alpha}}{\lambda_{\alpha}}\}. 
\end{equation} 

\begin{rem}
\label{rem:ab}
The invariants $a(L)$ and $b(L)$ may be computed even if $X$ is not smooth. 
Consider an equivariant resolution of singularities $\tilde{X}\ra X$, and let $\tilde{L}$ be the pullback of $L$ to $\tilde{X}$. 
Put
$$
a(L):=a(\tilde{L}), \quad \quad b(L):=b(\tilde{L}).
$$ 
A basic result is that this does not depend on the chosen resolution (see, e.g., \cite[Section 2]{balanced}. 
\end{rem}

The following proposition has been established in \cite{balanced}:

\begin{prop}
\label{prop:subgroup}
Let $M\subsetneq G$ be a closed connected subgroup containing 
$H$ and let $Y$ be the closure of $H\backslash M$ in $X$. Then
$$
(a(-K_X|_Y), b(-K_X|_Y)) < (a(-K_X), b(-K_X)), 
$$
in  the lexicographic ordering. 
\end{prop}

\begin{rem}
\label{rem:abc}
This fails in the non-equivariant context, see \cite{BT-cubic} for a counterexample and \cite{BT} 
for a discussion of this ``saturation'' phenomenon.  
\end{rem}

\section{Heights and height integrals}
\label{sect:heights}

Let $F$ be a number field, $\A$ its ring of adeles, and $\A_{f}$ the subring of finite adeles.
Let $v$ be a place of $F$ and $F_v$ the corresponding completion; for nonarchimedian $v$ we let
$\mathfrak o_v$ denote the ring of $v$-integers and $\mathfrak m_v$ its maximal ideal.  

Let $X$ be a projective variety over $F$, $U\subset X$ a Zariski open subset with boundary 
$$
\cup_{\alpha\in \mathcal A} D_{\alpha} = X\setminus U
$$
being a normal crossings divisor. Here $D_{\alpha}$ are $F$-irreducible components, which could be
reducible over an algebraic closure $\bar{F}$ of $F$. For each $\alpha$ one can endow
the line bundle  $\mathcal O(D_{\alpha})$ with an adelic metric which allows to define 
local and global heights
\begin{equation}
\label{eqn:uu}
\sH_{D_{\alpha,v}} : U(F_v)\ra \R_{>0}, \quad \quad \sH_{D_{\alpha}}:=\prod_v \sH_{D_{\alpha,v}}.
\end{equation}
We recall the construction: 
Let $\Omega\subset X$ be a chart such that in $\Omega$ the divisor $D_{\alpha}$ 
is given by the vanishing of the function $x_{\alpha}$. 
For almost all places $v$  of $F$, and $u_v\in \Omega(F_v)$, the local height is given by
$$
\sH_{D_{\alpha,v}}(u_v)= |x_{\alpha}(u_v)|_v^{-1}. 
$$
At all other places, the height differs from ``the distance to the boundary'' 
function by a globally bounded function.

The heights in \eqref{eqn:uu} 
give rise to an {\em adelic height system}
$$
\begin{array}{rcl}
\oplus_{\alpha} \C^{\mathcal A} \times U(\A) & \stackrel{\sH}{\longrightarrow} &  \C \\
        (\sum s_{\alpha}D_{\alpha}, (u_v))   & \mapsto &  \prod_{\alpha} \prod_v \sH_{D_{\alpha}, v}(u_v)^{s_{\alpha}}
\end{array}
$$
which restricts to a Weil height, for $u\in U(F)$ and $(s_{\alpha})\in \Z^{\mathcal A}$.
See Section 2 of \cite{CLT-igusa} for more details on the construction. 
The geometric framework developed in Section 4 of \cite{CLT-igusa} allows to 
establish analytic properties of local and 
global integrals of the form
\begin{equation}
\label{eqn:u}
\int_{U(F_v)} \sH_v(\mathbf s,u_v)^{-1}\, \mathrm  d\tau_v, \quad  \int_{U(\A)} \sH(\mathbf s,u)^{-1}\, \mathrm d\tau,
\end{equation}
where $\tau_v$ and $\tau$ are certain Tamagawa measures defined in Section 2 of \cite{CLT-igusa}. 
Proposition 4.1.2 and Proposition 4.3.5 of \cite{CLT-igusa} provide meromorphic continuations
for integrals in \eqref{eqn:u}.

We will apply this theory in the setup of Section~\ref{sect:geometry}. 
Let $G$ be a connected semi-simple algebraic group over $F$, 
$H$ a closed connected reductive subgroup, and $X$ a smooth projective $G$-equivariant
compactification of the affine variety $X^{\circ}:=H\backslash G$ with boundary 
$$
\cup_{\alpha\in \mathcal A} D_{\alpha} = X\setminus G,
$$ 
which we assume to be a divisor with strict normal crossings. The divisors $D_{\alpha}$ can be equipped 
with an adelic metrization which defines local and global heights on $X^{\circ}(\A)$. 
Furthermore, $G$-equivariance implies that for all  but finitely many $v$, 
the local height functions $\mathsf H_v$ are right-invariant under
$G(\mathfrak o_v)$ (see, e.g., Section 3 in \cite{chambert-t02}). 
The local and global measures $\mathrm d\tau_v$ and $\mathrm d\tau$ coincide with  
suitably normalized Haar measures $\mathrm dx_v$ and $\mathrm dx$ on $X^\circ(\A)=(H\backslash G)(\A)$.

\begin{lem}
\label{lem:galois}
Let $G$ be a connected algebraic group defined over a field $F$ and $H$ a closed subgroup.
Let $X^{\circ}= H\backslash G$ and assume that 
\begin{equation}
\label{eqn:coho}
\Ker\left(\rH^1(F, H) \to \rH^1(F, G)\right) =0.
\end{equation}
Then 
$$
X^{\circ}(F) = H(F)\backslash G(F). 
$$
\end{lem}

\begin{proof} 
Consider the sequence 
$$
1\ra H\ra G\ra H\backslash G\ra 1,
$$
and the corresponding long exact sequence in Galois cohomology. 
See, e.g., \cite[Chapter 1, Section 5.4]{serre}. 
\end{proof}

\begin{coro}
\label{coro:gn} 
Let $H$ be a connected algebraic group defined over a field $F$, acting diagonally on $G:=H^r$
Then 
$$
X^\circ(F) = H(F)\backslash G(F).
$$
In particular, if $F$ is a number field, then 
\begin{equation}
\label{eqn:points}
X^\circ(F_v)=H(F_v)\backslash G(F_v)\quad \text{ and } \quad X^\circ(\A)=H(\A)\backslash G(\A)
\end{equation}
\end{coro}

\begin{proof}
We have to show that 
$$
\Ker\left(\rH^1(F, H) \to \rH^1(F, G)\right) =0.
$$ 
Given a cocycle $c \in \rZ^1(F, H)$, $c: \Gal(\overline{F} /F) \to H (\overline{F})$, suppose that it is 
a coboundary in $\rZ^1(F, G)$. This means that $c(\sigma) = h^{-1} h^\sigma$ for $h = (h_1, \ldots, h_r) \in G$. 
Then $c(\sigma) = (h_1^{-1} h_1^\sigma, \ldots, h_r^{-1}h_r^\sigma)$. Since $c(\sigma) \in H(\overline{F})$, we have 
$$
h_i^{-1} h_i^\sigma = h_1^{-1} h_1^\sigma, \quad \forall i. 
$$
Hence $c(\sigma) = (h_1^{-1} h_1^\sigma, \ldots, h_1^{-1} h_1^\sigma)$, and as a result $c$ is a coboundary in $\rZ^1(F, H)$, as claimed. 
\end{proof}

The following theorem generalizes Theorem 7.1 of \cite{STT}. 

\begin{thm}
\label{thm:hecke}
Let $G$ be a connected semi-simple algebraic group and $H\subset G$ a closed subgroup, 
defined over a number field $F$, 
satisfying the vanishing condition \eqref{eqn:coho} for $F$ and all of its completions. 
Let $X$ be a smooth projective equivariant compactification of $X^{\circ}=H\backslash G$
with normal crossing boundary $\cup_{\alpha\in \mathcal A} D_{\alpha}$ and 
$$
\sH : \C^{\mathcal A}\times X^{\circ}(\A)\ra \C
$$ 
an adelic height system.   

For each automorphic character $\chi: G(\A)\ra \mathbb S^1$, trivial on $H(\A)$, there exist
a subset $\mathcal A(\chi)\subseteq \mathcal A$ and 
a function $\Phi_{\chi}$,  
holomorphic and bounded in vertical strips 
for $\Re(s_{\alpha}) > \kappa_{\alpha}-\epsilon$, for some $\epsilon>0$, 
such that for $\mathbf s=(s_{\alpha})$ in this domain one has
$$
\int_{X^\circ(\A)} \sH(\mathbf s, x)^{-1} \chi(x)\, \mathrm d x  = 
\prod_{\alpha\in \mathcal A(\chi)} \zeta_F(s_{\alpha}-\kappa_{\alpha}+1) 
\prod_{\alpha\notin\mathcal A(\chi)} \mathsf L(s_{\alpha}-\kappa_{\alpha}+1,\chi \circ \xi_{\alpha}) \cdot \Phi_{\chi}(\mathbf s), 
$$
where $\mathsf L$ are Hecke $\mathsf L$-functions. 
Moreover, $\mathcal A(\chi) = \mathcal A$ if and only if 
$\chi$ is trivial. 
\end{thm}

\begin{proof}
Using Corollary~\ref{coro:gn}, we rewrite the integral as
$$
\prod_v \int_{H(F_v)\backslash G(F_v)} \sH_v(\mathbf s, x_v)^{-1} \chi_v(x_v)\, \mathrm dx_v.
$$
For simplicity, we assume that the boundary divisors $D_{\alpha}$ are 
geometrically irreducible, otherwise, we need 
to work with Galois orbits as in \cite{STT}. 
We can ignore finitely many places, 
as they do not affect the poles of the Euler product (see, e.g., Section 4 of \cite{CLT-igusa}). 
At the remaining places, local integrals are computed in 
local analytic charts $\Omega_{A,v}$, labeled by boundary strata 
$$
D_A^{\circ}:=D_A\setminus  \cup_{A'\supsetneq A} D_{A'}, \quad 
D_A:=\cap_{\alpha\in A}D_{\alpha},
$$
with $A\subseteq \mathcal A$.
Observe that, 
\begin{itemize}
\item 
on charts with $|A|\ge 2$ we can replace $\chi$ by 1, these terms 
will not contribute to the leading poles of the Euler product (see, e.g., Section 9 of \cite{chambert-t02}); 
\end{itemize}

Using the $G(\mathfrak o_v)$-invariance of the local height functions, for almost all $v$, 
we may write the local height integrals as follows:
\begin{equation}
\label{eqn:local}
\int_{H(\mathfrak o_v)\backslash G(\mathfrak o_v)}  \chi_v(x_v) \, \mathrm d\mu_v  
+\sum_{\alpha\in \mathcal A} \int_{\Omega_{\alpha,v}} \sH_v(\mathbf s, x_v)^{-1}\chi_v(x_v) \, \mathrm d\mu_v +ET,
\end{equation}
where $ET$ is the error term, which for $\Re(s_{\alpha}) > \kappa_{\alpha} -\epsilon$, 
for all $\alpha\in \mathcal A$ and 
some $\epsilon >0$,  can be bounded by
$$
ET =  \frac{1}{q_v^{1+\delta}},  
$$
for some $\delta=\delta(\epsilon)>0$.
Here $q_v$ is the order of the residue field at $v$  
and $\mathrm d\mu_v$ is an appropriately normalized local Tamagawa measure.

To compute the local integrals on the charts $\Omega_{\alpha,v}$, 
we may assume that we are given rational functions $x_{\alpha}\in F(X)^\times$ 
and Zariski open charts $U_{\alpha}\subset X$ over $F$ such that in $U_{\alpha}$ 
the divisor $D_{\alpha}$ is given by the vanishing 
of $x_{\alpha}$. Let 
$$
\xi_{\alpha}:\mathbb G_{m}\ra G
$$ 
be a 1-parameter subgroup as in Lemma~\ref{lemm:1-par} 
so that the generic point of $D_{\alpha}$ is 
the limit of $\xi_{\alpha}(t)$, for $t\ra 0$, so that we may write,  
\'etale locally, $U_{\alpha}=Z_{\alpha}\times \mathbf{A}^1$, with $\mathbb G_m\hookrightarrow \mathbf A^1$. 
(A different choice of 1-parameter subgroups 
will not affect the poles of the local integrals below and
thus the poles of the Euler product.)
Expressing a $g_v\in H(F_v)\backslash G(F_v)\cap \Omega_{\alpha,v}$ 
as $g_v=(z_v,t_v)$, with $t_v\neq 0$, we have
$$
x_{\alpha}(g_v)= u_v(z_v,t_v) \cdot t_v,
$$
where $u_v(z_v,t_v)\in \mathfrak o_v^\times$ is a unit, for almost all $v$. 
On the other hand, we have
$$
\xi_{\alpha}(t_v)=(z_{\alpha,v}(t_v),t_v). 
$$
Thus 
$$
\lim_{t_v\ra 0} \frac{x_{\alpha}(\xi_{\alpha}(t_v))}{t_v} = w_{\alpha,v}(z_v), 
$$
where $w_{\alpha,v}(z_v)\in \mathfrak o_v^\times$ is a unit.  

Each automorphic character 
$$
\chi : G(\A) \ra \mathbb S^1
$$
and each 1-parameter subgroup $\xi_{\alpha}$ 
give rise to a Hecke character 
$$
\chi_{\alpha}:=\chi\circ \xi_{\alpha} : \mathbb G_m(F) \backslash \mathbb G_m(\A) \ra \mathbb S^1.  
$$
For almost all $v$, in the chart $\Omega_{\alpha,v}$ and for $t_v=t_{\alpha,v}$, with $|t_v|_v$ sufficiently small, we have: 
$$
\sH_v(\mathbf s, (z_v,t_v))^{-1}\chi_v((z_v,t_v)) =|t_v|^{s_{\alpha}-\kappa_{\alpha}} \chi_{\alpha,v}(t_v). 
$$
 The local integrals \eqref{eqn:local} take the form
$$ 
1+ \sum_{\alpha\in \mathcal A} 
\int_{\mathfrak m_v} |t_{v}|_v^{s_{\alpha}-\kappa_{\alpha}   +im_{\alpha, v}}\, \mathrm d t_v
\cdot \frac{1}{q_v^{{\rm dim}(X)-1}} + ET,  
$$
where
\begin{itemize}
\item 
$\mathrm dt_v$ a normalized Haar measure on $\mathfrak o_v$;
\item the local character is given by 
$$
\chi_{\alpha,v}(t_v)= |t_v|^{im_{\alpha,v}}, \quad \text{ for some } \quad m_{\alpha,v} \in \R.
$$
\end{itemize} 
(See the computations on p. 444 of \cite{chambert-t02}.)
We obtain
$$
\int_{H(F_v)\backslash G(F_v)} \sH_v(\mathbf s, x_v)^{-1}\chi_v(x_v) \, \mathrm dx_v = 1+
\left(\sum_{\alpha\in \mathcal A} \frac{1}{q_v^{s_{\alpha}-\kappa_{\alpha}+1+im_{\alpha,v}}}\right) + 
O(q_v^{-(1+\delta)}),
$$  
for some $\delta>0$, provided $\Re(s_{\alpha}-\kappa_{\alpha}+1)>\epsilon'$, for some $\epsilon' >0$. 
The corresponding Euler product is regularized by 
$$
\prod_{\alpha\in \mathcal A} \mathsf L(s_{\alpha}-\kappa_{\alpha}+1,\chi_{\alpha}). 
$$

It remains to observe that if 
$\chi : G(\A)\ra \mathbb S^1$ is an automorphic character such that 
$\chi_{\alpha}=1$, for all $\alpha\in \mathcal A$, then $\chi = 1$. 
This is analogous to \cite[Proposition 8.6]{STT}. 
\end{proof}

\begin{lem}[Well-roundedness of adelic height balls I]
\label{lemm:balls} 
Let $L$ be a class in the interior of the cone of effective divisors and $\sH$ the associated height. 
Then the corresponding height balls 
$$
B_T=\{x\in X^\circ(\A)\,:\, \sH(x)<T\}
$$
are well-rounded, i.e., 
$$
\lim_{\kappa\to 1^+}\limsup_{T\ra \infty} \frac{ \vol(B_{\kappa T})  -  \vol(B_{\kappa^{-1}T})}{\vol(B_T)} = 0.
$$ 
\end{lem}

\begin{proof}
This is a corollary of the main theorem of \cite{CLT-igusa}, which establishes  
analytic properties of the height integrals  of the form 
$$
\int_{X^{\circ}(\A)} \sH(\mathbf s, x)^{-1}\, \mathrm dx. 
$$
The main pole comes from the Euler product defined by the adelic integral. 
We only need to show
that the local integrals are {\em holomorphic} for $\Re(\mathbf s)$ 
in a neighborhood of the shifted cone $\Lambda_{\rm eff}(X)+K_X\in \Pic(X)_{\R}$. 
This is immediate if the metrization is {\em smooth}.
In this case  Proposition 4.3.5 of \cite{CLT-igusa} shows that 
$$
\int_{X^{\circ}(\A)} \sH(\mathbf s, x)^{-1}\, \mathrm dx = 
\prod_{\alpha \in \mathcal A} \zeta_{F_{\alpha}} (s_{\alpha}-\kappa_{\alpha}+1) \cdot \Phi(\mathbf s),
$$
where $\Phi$ is a function 
holomorphic and bounded in vertical strips 
in the tube domain $\Re(s_{\alpha})>\kappa_{\alpha}-\epsilon$, for some $\epsilon>0$. 
 
Then we restrict to the line $sL$ and apply a Tauberian theorem. Since the Euler product 
is regularized by Dedekind zeta functions, which satisfy standard convexity 
bounds in vertical strips, a suitable Tauberian theorem gives an expansion
$$
\vol(B_T) = c \, T^{a(L)}P(\log(T)) + O(T^{a(L)-\delta}),
$$ 
where $c>0$,  $P$ is a monic polynomial of degree $b(L)-1$ and the implied constants and $\delta>0$ are explicit. 

The general case follows from the smooth case: 
for any constant $r>1$ there exists a smooth metrization 
such that the corresponding height function $\sH'$ satisfies
$$
r^{-1} \sH' < \sH < r\sH'.
$$
Thus for any $T>0$, we have
$$
B'_{r^{-1}T}\subseteq B_T \subseteq B'_{rT}
$$ 
so that 
$$
\limsup_T \frac{\vol(B_{\kappa T}) - \vol(B_{\kappa^{-1} T})}{\vol(B_T)} 
\leq \limsup_T \frac{\vol (B'_{r \kappa T}) - \vol (B'_{r^{-1} \kappa^{-1} T})}{\vol(B'_{r^{-1}T})},
$$
which can be made arbitrarily small by taking $r$ and then $\kappa$ close enough to 1.
This implies that the height balls are well-rounded. 
\end{proof}

\section{Intermediate subgroups} 
\label{sect:intermediate}

Let $H$ be a connected simple algebraic group defined
over an algebraically closed field of characteristic zero. 
Let $Z(H)$ be the center of $H$. 
For $n\in \N$, let 
$$
H[n] = H \times \cdots \times H
$$
be the $n$-th fold product of $H$.  
Let $\Delta[n]$ be the diagonal in $H[n]$, i.e., 
the subset of $H[n]$ consisting of elements of the form 
$(x,x,\ldots, x)$ with $x\in H$. 
The symmetric group $\mathfrak S_n$ acts on 
$H[n]$ by permuting coordinates. 
For $\sigma \in \mathfrak S_n$ and $M\subset H[n]$, 
let $\sigma(M)\subset H[n]$ be the image of $M$ under $\sigma$. 
We call subgroups $M, N$ of $H[n]$ \emph{permutation equal} 
if there is a $\sigma \in \mathfrak S_n$ such that $M = \sigma(N)$; such subgroups
are clearly isomorphic.  
This following proposition is used in the proof of the multiple mixing property
in Section \ref{sect:mixing}.

\begin{prop}
\label{prop:equal}
Let $H$ be a connected simple algebraic group and
$M$ a connected algebraic group such that 
$$
\Delta[n]\subseteq M\subseteq H[n].
$$
Then there exist $n_1, \ldots, n_k\in \N$ 
such that $\sum_{i=1}^k n_i = n$ and $M$ is permutation equal to
$$
\Delta[n_1] \times \cdots \times \Delta[n_k].
$$
\end{prop}


The remainder of this section is devoted to a proof of Proposition~\ref{prop:equal}.
The main step is the following version of Goursat's lemma:

\begin{lem}
\label{lem:start} 
Let $\underline{x}_r = (x_1, \ldots, x_r) \in H[r]$ be such that for all $i$, 
we have $x_i \notin Z(H)$ and for all $i \ne j$, we have $x_i x_j^{-1} \notin Z(H)$. 
Let $L_r\subseteq H[r]$ be the smallest subgroup containing 
$$
\Gamma_r:= \{(\delta x_1 \delta^{-1}, \ldots, \delta x_r \delta^{-1} )\, |\, \delta \in H\}.
$$
Then $L_r=H[r]$.
\end{lem}

\begin{proof}
We assume that
$Z(H)=1$
and proceed by induction on $r$. Note that $\Gamma_1$ is nontrivial and that it is 
closed under conjugation so that the closed subgroup of $H=H[1]$ generated by $\Gamma_1$ is normal. 
Since $H$ is simple,  $L_1=H$. 

Assume the statement for $r>1$. Let $L_{r}$ be the subgroup corresponding to 
the vector $\underline{x}_{r} : = (x_1, \ldots, x_{r})$, we assume that $L_{r} = H[r]$. 
Clearly, $L_{r}$ is the projection of 
$L_{r+1}$ onto the first $r$ entries. Applying the case $r=1$, 
we deduce that the projection of $L_{r+1}$ onto the last entry is equal to $H$. 
Suppose that there is an element $h \in H[r]$ such that for two distinct elements $u,v \in H$, 
we have $(h, u ) \in L_{r+1}$ and $(h, v) \in L_{r+1}$. 
Then $(e_{r}, uv^{-1}) \in L_{r+1}$,
where $e_{r}$ denotes the vector in $H[r]$ consisting of identity elements in every entry. 
Again by the case when $r=1$, we see that $\{e_{r}\}\times H \subset L_{r+1}$. Since the projection onto the 
first $r$ coordinates is surjective, $L_{r+1} = H[r] \times H = H[r+1]$,
as required.  
It remains to rule out the case when 
for every $h \in H[r]$ there is a unique $u:= u(h)$ such that $(h, u(h)) \in L_{r+1}$. 
It follows from the uniqueness that the map $\varphi:h \mapsto u(h)$ is a homomorphism 
$H[r] \to H$, and
$$
L_{r+1}=\{(h, \varphi(h))\, |\, h\in H[r]\}.
$$
Moreover, $\varphi$ is surjective.
By construction, if $(h, \varphi(h)) \in L_{r+1}$, then for any $\delta \in H$, we have 
$$
(\delta_{r} h \delta_{r}^{-1},\delta \varphi(h) \delta^{-1}) \in L_{r+1},
$$
where $\delta_r$ denotes the vector in $H[r]$ with $\delta$ in every entry. 
It follows from uniqueness that 
$$
\varphi(\delta_{r} h\delta_{r}^{-1}) = \delta \varphi(h) \delta^{-1}.
$$
Hence, $\delta^{-1} \varphi(\delta_{r})$ commutes with $\varphi(h)$ for every $h\in H[r]$.
Since $\varphi$ is surjective, we see that $\varphi(\delta_r)=\delta$ for every $\delta\in H$.
Now let $\delta=\varphi(\underline{x}_r)$ and $\underline{y}_r=\underline{x}_r\delta_r^{-1}$.
Then $\underline{y}_r\in \ker(\varphi)$, and $y_iy_j^{-1}=x_ix_j^{-1}\ne e$ for $i\ne j$.
If $y_i\ne e$ for all $i$, then it follows from the inductive assumption 
applied to $\underline{y}_r$ that $\ker(\varphi)=H[r]$, which contradicts surjectivity of $\varphi$.
Hence, there exists an $\ell=1,\ldots,r$ such that $y_{\ell}=e$. Since $y_i\ne y_j$ for $i\ne j$,
such $\ell$ is unique. Again, by the inductive assumption,
$H[\ell-1]\times \{e\}\times H[r-\ell]\subseteq \ker(\varphi)$. We conclude that
$$
L_{r+1}=\{(h_1,\delta,h_2,\delta)\, |\, h_1\in H[\ell-1], \delta\in H, h_2\in H[r-\ell]\}.
$$
On the other hand, $\underline{x}_{r+1}\in L_{r+1}$ and $x_\ell\ne x_{r+1}$, a contradiction.

To treat the general case, we consider the projection
$\pi: H[r]\to \bar H[r]$, where $\bar H=H/Z(H)$. Since $Z(\bar H)=1$, it follows from
above that $\pi(L_r)=\bar H[r]$ and $H[r]=L_r\cdot \ker(\pi)$.
Thus $L_r$ has finite index in $H[r]$, and hence $L_r=H[r]$.
\end{proof}

\begin{defn}\label{def:adm}
Let $r \leq n$ be integers. An {\it admissible embedding} of $H[r]$ in
$H[n]$ is the obvious map
$$
\begin{array}{rcl}
H[r] & \to &  \Delta[n_1] \times \cdots \times \Delta[n_r] \\
(h_1,\ldots, h_r) & \mapsto & (\underbrace{h_1, \ldots, h_1}_{n_1}, \underbrace{h_2,\ldots,h_2}_{n_2}, \ldots , 
\underbrace{h_{r}, \ldots, h_r}_{n_r}),
\end{array} 
$$
with $\sum_i n_i =n$, followed by a permutation of coordinates. An
{\it admissible subgroup} of $H[n]$ is the image of an
admissible embedding.
\end{defn}

We note that there is a one-to-one correspondence between admissible subgroups
of $H[n]$ and partitions of the set $\{1,\ldots,n\}$.


\begin{defn}
Given $r \leq n$, we say an element $\underline{x} \in H[n]$ is of rank $\leq
r$, if $\underline{x} \in \iota (H[r])$ for some admissible embedding $\iota$.
We say $\underline{x}$ is of rank $r_0$, written $r(\underline{x})=r_0$, if $r_0$ is the
smallest number $r$ such that $\underline{x}$ is of rank $\leq r$.
\end{defn}


It is clear that for every $\underline{x} \in
H[n]$, $r(\underline{x}) \leq n$.  Note that if $\underline{x} \in H[n]$ and $\underline{\delta} \in
\Delta[n]$ then
$$
r(\underline{x}\cdot \underline{\delta}) = r(\underline{x}), \quad \text{ for } \underline{x} \in H[n].
$$

\begin{proof}[Proof of Proposition~\ref{prop:equal}] 
A reformulation of the
statement of the proposition is that if $M$ is a connected
subgroup of $H[n]$ satisfying
$$
\Delta[n] \subset M \subset H[n],
$$
then $M$ is admissible. 
Since the isogeny $\pi: H[r]\to \bar H[r]$, where $\bar H=H/Z(H)$,
define a bijection between closed connected subgroup of $H[r]$ and $\bar H[r]$,
it is sufficient to prove the claim assuming that $Z(H)=1$.

Let $r= \max_{\ux \in M} r(\ux)$, and let
$\underline{x}$ be an element of $M$ which realizes this maximum. As $\Delta[n]
\subset M$, we may assume that no entry of $\underline{x}$ is equal to identity.
After rearranging the coordinates, if necessary, we may
assume that
$$
\ux = (x_1, \ldots, x_1, x_2, \ldots, x_2, \cdots, x_r, \ldots, x_r) \in \Delta[n_1] \times \cdots \times \Delta[n_r]
$$
where $x_i x_j^{-1}\ne e$ for $i\ne j$. Then since $\Delta[n]\subset M$,
it follows from Lemma~\ref{lem:start} that
$$  
N:=\Delta[n_1] \times \cdots \times \Delta[n_r] \subseteq M.
$$
To prove the proposition it suffices to establish that 
$N = M$.
Indeed, if $M$ were larger than $N$, multiplying a generic element of $N$ 
by an element of $M\setminus N$ we would get an element $\ux'$ with $r(\ux')>r(\underline{x})$, 
a contradiction.

\end{proof}

\section{Multiple mixing}
\label{sect:mixing}

Let $H$ be a connected semi-simple algebraic group defined over a number field $F$. 
The aim of this section is to prove the multiple mixing property for
the adelic homogeneous space 
$H(F)\backslash H(\A)$.
However, when the group $H$ is not simply connected,
$\mathsf L^2(H(F)\backslash H(\A))$ contains nontrivial characters, and the multiple mixing property
holds only on a subset $Y_W$ of $Y$, which we now introduce.
Let $\pi: \tilde{H} \to H$ be the universal cover of $H$ and
$W$ a compact subgroup of $H(\A)$ such that $W \cap H(\A_f)$ is open in $H(\A_f)$. 
We set
\begin{equation}\label{eq:h_w}
H_W := H(F) \pi(\tilde{H}(\A))W.
\end{equation}
By \cite{GO}, Corollary 4.10, $H_W$ is a normal 
closed co-abelian subgroup of finite index in $H(\A)$. We 
consider the homogeneous space
$$
Y_W := H(F) \backslash H_W,
$$
equipped with the normalised Haar measure $\mathrm dy$. Let 
$\mathsf C_c(Y_W)^W$ denote the space of continuous compactly supported and $W$-invariant functions 
on $Y_W$.

The following theorem is an adelic version of the multiple mixing of S.~Mozes \cite{moz}.

\begin{thm}[multiple mixing]
\label{thm:mult}
Let $H$ be a connected simple group over $F$ and 
  $$
  \{ (b_1^{(n)}, \ldots, b_r^{(n)})\}_{n\in \N} \subset H_W[r]= H_W \times \cdots \times H_W
  $$
a sequence such that for all $i\ne j$, 
$$
\lim_{n\ra \infty} (b_i^{(n)})^{-1}b_j^{(n)} = \infty\quad\hbox{in $H_W$.}
$$ 
Then for all $f_1, \ldots, f_r \in \mathsf C_c(Y_W)^W$, we have
\begin{equation}
\label{eqn:seq}
\lim_{n\ra \infty}  \int_{Y_W} f_1(yb_1^{(n)}) \cdots f_r(yb_r^{(n)}) \, \mathrm d y 
= \left(\int_{Y_W} f_1 \,\mathrm d y\right) \cdots \left(\int_{Y_W} f_r \, \mathrm d y\right).
\end{equation}
\end{thm}

The proof of Theorem \ref{thm:mult} is based on an interpretion of integrals in \eqref{eqn:seq}
as a sequence of probability measures supported on $Y_W\times \cdots Y_W$ and on an analysis
of their limit behaviour using the theory of unipotent flows on adelic spaces
developed in \cite{GO}. The main technical tools are
a partial case of Theorem 1.7 of \cite{GO} combined with the description of intermediate
subgroups from Section \ref{sect:intermediate}.

For $g\in G(\A)$ and a measure $\nu$ on $G(F)\backslash G(\A)$, let
$g\cdot \nu$ be the push-forward of $\nu$
via the right multiplication by $g$.  

\begin{thm}[\cite{GO}, Theorem 1.7]
\label{thm:go}
Let $G$ be a connected semi-simple algebraic group defined over a number field $F$,
$H$ a connected semi-simple subgroup defined over $F$, and 
$V$ a compact subgroup of $G(\A)$ such that $V \cap G(\A_f)$ is open in $G(\A_f)$. 
Let $\nu_L$ be the unique $\tilde{L}(\A)$-invariant probability measure supported on 
$G(F) \pi(\tilde{L}(\A)) \subset G(F) \backslash G_V$, and let $g^{(n)}$ be a sequence in 
$G(F) \pi(\tilde{G}(\A)) \subset G_V$. Then
\begin{enumerate}
  \item If the centralizer of $L$ in $G$ is anisotropic over $F$, then the sequence of 
measures $\{g^{(n)} \cdot\nu_L\} $ is precompact in the weak$^*$ topology.
  \item Suppose that a probability measure $\mu$ on $G(F) \backslash G_V$ is a limit 
of the sequence 
$\{g^{(n)}\cdot \nu_L\}$ in the weak$^*$ topology. Then there exists a connected algebraic
subgroup $M$ of $G$ defined over $F$ such that:
      \begin{itemize}
        \item $\delta^{(n)} L (\delta^{(n)})^{-1} \subset M$ for some sequence $\delta^{(n)} \in G(F)$,
        \item for some sequence $l^{(n)} \in \pi(\tilde{L}(\A))$, $\delta^{(n)} l^{(n)}g^{(n)} \to g \in 
\pi(\tilde{G}(\A))$,
      \end{itemize}
and the limit measure $\mu$ can be described as follows: there is a finite index normal 
subgroup $M_0$ of $M(\A)$,
containing $M(F) \pi(\tilde{M}(\A))$, such that
for all $f \in \mathsf C_c(G(F) \backslash G_V)^V$,
            $$
            \int_{G(F) \backslash G_V} f \, \mathrm d \mu = \int_{G(F) \backslash G_V} 
f \, \mathrm d(g\cdot\nu_{M_0}),
            $$
where $\nu_{M_0}$ denotes the unique invariant probability measure supported on 
$G(F) M_0 \subset G(F) \backslash G_V$.
\end{enumerate}
\end{thm}

\begin{proof}[Proof of Theorem~\ref{thm:mult}]
We apply Theorem \ref{thm:go} to the groups
\begin{align*}
G &= H[r] = H \times \cdots \times H,\\
L &= \Delta[r]=\{(h, \ldots, h)\, |\,  h \in H\},\\
V &= W\times \cdots \times W.
\end{align*}
Since $H(F) \pi(\tilde{H}(\A))$ is a normal subgroup of $H_W$ (see \cite{GO}, Section~4)
and $W$ is compact, the normalized Haar measure on $Y_W$ can be written as
\begin{equation}\label{eq:measure}
\int_{Y_W} f\, \mathrm d y=\int_{Y_W\times W} f(uw)\, \mathrm d\nu_H(u) \mathrm d w,\quad 
f\in \mathsf C_c(Y_W),
\end{equation}
where $\nu_H$ is 
the unique $\tilde{H}(\A)$-invariant probability measure on 
$H(F) \pi(\tilde{H}(\A)) \subset Y_W$, and $\mathrm d w$ is the probability invariant measure on $W$.
Therefore, 
$$
\int_{Y_W} f_1(xb_1^{(n)}) \cdots f_r(x b_r^{(n)}) \, \mathrm d x = 
\int_{Y_W\times W} f_1(u w b_1^{(n)}) \cdots f_r ( u w b_r^{(n)}) 
\, \mathrm d\nu_H(u) \mathrm d w.
$$
If we show that for every fixed $w\in W$, we have 
$$
\lim_{n\to\infty} \int_{Y_W}  f_1(u w b_1^{(n)}) \cdots f_r ( u w b_r^{(n)}) 
\, \mathrm  d\nu_H(u) = \left(\int_{Y_W} f_1\,  \mathrm d y\right) \cdots \left(\int_{Y_W} f_r\,  \mathrm d y\right),
$$
then the theorem would follow from the Lebesgue dominated convergence theorem. 

We write $wb_i^{(n)}=s_i^{(n)}w_i^{(n)}$ for $s_i^{(n)}\in H(F) \pi(\tilde{H}(\A))$
and $w_i^{(n)}\in W$. Since the functions $f_i$ are assumed to be
$W$-invariant, 
$$
\int_{Y_W} f_1(uw b_1^{(n)}) \cdots f_r(uw b_r^{(n)}) \, \mathrm d\nu_H(u)
=\int_{Y_W} f_1(u s_1^{(n)}) \cdots f_r(u s_r^{(n)}) \, \mathrm d\nu_H(u).
$$
Since $W$ is compact, we have
\begin{equation}\label{eq:s}
(s_i^{(n)})^{-1}s_j^{(n)} =w_j^{(n)}\cdot (b_i^{(n)})^{-1}b_j^{(n)}\cdot (w_j^{(n)})^{-1}\to \infty
\end{equation}
for all $i\ne j$. We set
$$
g^{(n)}=(s_1^{(n)}, \ldots, s_r^{(n)})\in G(F) \pi (\tilde{G}(\A)).
$$
Then
$$
\int_{Y_W} f_1(u s_1^{(n)}) \cdots f_r(u s_r^{(n)}) \, \mathrm d\nu_H(u)
= \int_{G(F) \backslash G_V} (f_1 \otimes \cdots \otimes f_l) \,  \mathrm d (g^{(n)}\cdot \nu_L).
$$
Now it remains to determine the limit points of the sequence of measures
$g^{(n)}\cdot \nu_L$ in the weak$^*$ topology. We first note that
the centraliser of $L$ in $G$ 
is equal to $Z(H) \times \cdots \times Z(H)$. Hence, by Theorem~\ref{thm:go}(1), 
the sequence of measures $g^{(n)}\cdot \nu_L$ is precompact. 
Let $\mu$ be a probability measure on $G(F)\backslash G_V$
which is a limit point of this sequence.
The measure $\mu$ is described in Theorem \ref{thm:go}(2). In particular, we obtain that
there exist a connected algebraic subgroup $M$ of $G$ 
and a sequence $\delta^{(n)}\in G(F)$ such that
$$
L \subseteq (\delta^{(n)})^{-1} M \delta^{(n)} \subseteq G
$$
From the classification of intermediate subgroups in Proposition~\ref{prop:equal}, we deduce that
$$
M=\delta^{(n)} N_n (\delta^{(n)})^{-1},
$$
where $N_n$ is an admissible subgroup (in the sense of Definition \ref{def:adm}).

We claim that $M=G$. Indeed, suppose that $M\subsetneq G$. 
Since the number of admissible subgroups is finite,
we may assume, after passing to a subsequence, that
$N_n=N\subsetneq G$ is independent of $n$.
Then there exist indices $i \ne j$ such that 
for the corresponding projection map  $\pi_{ij}: G \to H \times H$,
we have $\pi_{ij}(N)=\Delta$, where $\Delta$ denotes the diagonal subgroup in $H\times H$.
Let $\delta = \delta^{(1)}$ and $\sigma^{(n)} = \delta^{-1} \delta^{(n)}$. Since
$$
\delta^{(1)} N (\delta^{(1)})^{-1}=\delta^{(n)} N (\delta^{(n)})^{-1},
$$
we obtain
$$
  \pi_{ij}(\sigma^{(n)})\, \Delta\, \pi_{ij}(\sigma^{(n)})^{-1} = \Delta,
$$
  and
  $$
  (1, (\sigma_i^{(n)})^{-1} \sigma_j^{(n)})\,\Delta\, (1, (\sigma_i^{(n)})^{-1} \sigma_j^{(n)}) = \Delta.
  $$
  This implies that
  $$
  z_n := (\sigma_i^{(n)})^{-1} \sigma_j^{(n)} \in Z(H).
  $$
By Theorem \ref{thm:go}(2), we also know that 
there exist $l^{(n)}\in\pi(\tilde{L}(\A))$ such that
the sequence $\delta^{(n)} l^{(n)} g^{(n)}$ converges. 
Then the sequence $\sigma^{(n)} l^{(n)} g^{(n)}$ converges too, and in particular,
$$
(\sigma_i^{(n)} l_i^{(n)} s_i^{(n)})^{-1} (\sigma_j^{(n)} l_j^{(n)} s_j^{(n)})
$$
converges. Since $l_i^{(n)} = l_j^{(n)}$ and $z_n \in Z(H)$, we obtain
\begin{align*}
(\sigma_i^{(n)} l_i^{(n)} s_i^{(n)})^{-1} (\sigma_j^{(n)} l_j^{(n)} s_j^{(n)})
& = (s_i^{(n)})^{-1}(l_i^{(n)})^{-1}(\sigma_i^{(n)})^{-1} \sigma_j^{(n)} l_j^{(n)} s_j^{(n)}\\
& =(s_i^{(n)})^{-1}(l_i^{(n)})^{-1}z_n l_j^{(n)} s_j^{(n)}\\
& = z_n^{-1} (s_i^{(n)})^{-1}s_j^{(n)}.
\end{align*}
Since $z_n$ runs over the finite set $Z(H)$,
it follows that $(s_i^{(n)})^{-1}s_j^{(n)}$ converges, which is a contradiction.
This proves that $M=G$.

By the last statement of Theorem~\ref{thm:go}, 
there is a finite index subgroup $M_0\subseteq M(\A)=G(\A)$, 
containing $G(F) \pi(\tilde{G}(\A))$, and $g\in\pi(\tilde G(\A))$ such that
for all $f \in \mathsf C_c(G(F) \backslash G_V)^V$,
            $$
            \int_{G(F) \backslash G_V} f \, \mathrm d \mu = \int_{G(F) \backslash G_V} 
f \, \mathrm d(g\cdot\nu_{M_0}),
            $$
Since $G(F) \pi(\tilde{G}(\A))$ is a normal coabelian subgroup of $G_V$
(see \cite{GO}, Section 4), $M_0$ is also normal coabelian.
As in (\ref{eq:measure}), the normalised Haar measure $\mathrm d z$ on $G(F)\backslash G_V$
is given by
$$
\int_{G(F)\backslash G_V} f\, \mathrm d z=\int_{G(F)\backslash G_V\times V} f(uv)\, \mathrm 
d\nu_{M_0}(u) \mathrm d v,\quad 
f\in \mathsf C_c(G(F)\backslash G_V),
$$
where $\mathrm d v$ is the normalised Haar measure on $V$.
For $f\in \mathsf C_c(G(F)\backslash G_V)^V$, using that $M_0$ is coabelian, we obtain
\begin{align*}
\int_{G(F)\backslash G_V} f\, \mathrm d z &=
\int_{G(F)\backslash M_0\times V} f(uvg)\, \mathrm d\nu_{M_0}(u) \mathrm d v\\
&=
\int_{G(F)\backslash M_0\times V} f(ugv)\, \mathrm  d\nu_{M_0}(u) \mathrm d v\\
&=
\int_{G(F) \backslash G_V} f \, \mathrm d(g\cdot\nu_{M_0}).
\end{align*}
This proves that every limit point of the sequence 
$g^{(n)}\cdot \nu_L$ is a probability measure which is equal to $\mathrm d z$
on $\mathsf C_c(G(F)\backslash G_V)^V$ which completes the proof of Theorem~\ref{thm:mult}. 
\end{proof}

\section{Counting rational points}
\label{sect:count}

Let $H$ be a connected simple algebraic group defined over a number field $F$,
$G=H^r$, and $X$ be a smooth projective equivariant compactification of 
$X^{\circ}:=H\backslash G$, where $H$ is embedded diagonally. 
Let $L$ be a line bundle on $X$
such that its class is in the interior of the cone of effective divisors $\Lambda_{\rm eff}(X)$. 
By Proposition~\ref{prop:picard}, $\Lambda_{\rm eff}(X)$ is spanned by 
the classes of boundary components $D_{\alpha}$   
of $X\setminus X^{\circ}$.  
In particular, 
$$
L=\sum_{\alpha\in \mathcal A} \lambda_{\alpha} D_{\alpha}, \quad \lambda_{\alpha}>0. 
$$
Let 
$$
\sH=\sH_{\mathcal L} : X^{\circ}(F)\ra \R_{>0}
$$
be a height corresponding to a suitable metrization of $L$ as in Section~\ref{sect:heights}.

\


The height function $\sH$ is invariant under a compact open subgroup $V$ of $G(\A_f)$,
and we may assume that $V=W\times\cdots \times W$ for a compact open subgroup $W$ of $H(\A_f)$.
We define the subgroups $G_V\subset G(\A)$ and $H_W\subset H(\A)$ as in (\ref{eq:h_w}).
The homogeneous space 
$$
X_V:=H_W\backslash G_V
$$
naturally embeds into $X^{\circ}(\A)$ as an open subset.

We equip $X^{\circ}(\A)$ with the Tamagawa measures $\mathrm d x$, defined as in Section 2 of \cite{CLT-igusa},
and define the height balls in $X_V$
by
$$
B_T=B_{T, \mathcal L} = \{ x \in X_V \, | \, \sH_{\mathcal L}(x) < T \}. 
$$

The following lemma, which is not hard to prove, is a crucial ingredient of the proof of our main theorem: 

\begin{lem}[Well-roundedness of adelic heigh balls II]  
 Let $\kappa>1$. There is a symmetric neighborhood $U$ of the identity in $G_V$ such that
\begin{equation}\label{eq:bt}
B_T\cdot U\subset B_{\kappa T}\quad\hbox{for all $T$.}
\end{equation}
\end{lem} 

\begin{lem}
\label{lem:volume}
Assume that the line bundle $L$ is in the interior of the effective cone. Then
$$
\vol(B_T) = c\cdot T^{a(L)}\log(T)^{b(L)-1}(1+o(1)) \quad \hbox{as $T\ra \infty$, }
$$ 
with $c>0$ and $a(L),b(L)$ as in (\ref{eqn:main}).
\end{lem}

\begin{proof}
Using a standard Tauberian argument (see, for instance, \cite{CLT-igusa}), it suffices to show that 
$$
\mathsf Z(s) = \int_{X_V} \mathsf H(x)^{-s} \, \mathrm d x
$$
has an isolated pole at $a(L)$ of order $b(L)$ and that it admits a meromorphic continuation 
to $\Re(s)>a(L)-\epsilon$, for some $\epsilon >0$. 
We recall (see \cite{GO}, Section 4) that $G_V$ is a normal closed coabelian subgroup of $G(\A)$.
Let $\mathfrak{X}$ be the set of characters of $G(\A)$ invariant under $H(\A)$ and $G_V$.  
By the finite abelian Fourier analysis, for $g\in G(\A)$, we have 
$$
\sum_{\chi \in {\mathfrak X}} \chi(g) = \begin{cases}
0 & g \not\in H(\A)G_V ; \\
[G(\A):H(\A)G_V] & g \in H(\A)G_V. 
\end{cases}
$$
Thus,
$$
\mathsf Z(s) = \frac{1}{[G(\A):H(\A)G_V]} 
\sum_{\chi \in \mathfrak{X}} \int_{H(\A)\backslash G(\A)} \mathsf H(x)^{-s} \chi(x) \, \mathrm d x. 
$$
The meromorphic continuation of 
$$
\int_{H(\A)\backslash G(\A)} \mathsf H(x)^{-s} \chi(x) \, \mathrm d x
$$
is the content of Theorem \ref{thm:hecke}.  We need to show that the highest order pole does not cancel out. This follows from a standard 
argument (see e.g. the proof of Theorem 6.4 of \cite{TT}).  
\end{proof}

\begin{defn}
\label{defn:diagonal}
Let $X$ be an equivariant compactification of $X^{\circ}=H\backslash G$. 
Let 
$$
H[n]_{i,j}:=\{ \underline{h}=(h_1,\ldots, h_n) \,|\,h_i=h_j \} \subset G=H[n]
$$
be the small diagonal subgroup and let $Y_{ij}\subset X$ be the induced compactification of 
$\Delta[n]\backslash H[n]_{i,j}$.
A line bundle $L$ on $X$ is called {\em balanced} if for every 
$i \ne j$ one has
$$
(a(L|_{Y_{ij}}), b(L|_{Y_{ij}})) < (a(L), b(L)), 
$$
in the lexicographic ordering. 
\end{defn}

\begin{rem}
This property fails in simple examples: $X=\mathbf P^3\times\mathbf P^3$ 
considered as an equivariant compactification of 
$\mathbb G_m^6$ or $\mathbb G_a^6$, or $\PGL_2\times \PGL_2$, 
with $L=(\lambda_1,\lambda_2)$ and $\lambda_1\neq \lambda_2$.   
\end{rem}

\begin{lem}
\label{lem:balanced}
Assume that the line bundle $L$ is balanced. Then, for every smooth adelic metrization of $L$,
every compact subset $K$ of $H_W$  and $i\ne j$, one has
\begin{equation}
\label{eqn:bal}
\frac{\vol (B_{T} \cap \{ x_i^{-1} x_j \in K\})}{ \vol (B_T)} \to 0\quad\hbox{as $T \to \infty$.}
\end{equation}
\end{lem}


\begin{proof}
Let $M\subset G=H^n$ be the subgroup defined by $x_i=x_j$. 
Lemma~\ref{lemm:balls} implies that, for $T\ra \infty$, one has 
\begin{eqnarray*}
\vol(B_T)                         & = & c\, T^{a(X,L)}\log(T)^{b(X,L)-1}(1+o(1))\\
\vol(B_T\cap  \{x_i^{-1} x_j = 1\} )& = & c'\,T^{a(Y,L|_Y)}\log(T)^{b(Y,L|_Y)-1}(1+o(1)),
\end{eqnarray*}
where $Y=Y_{ij}$ is the induced equivariant compactification of 
$$
Y^\circ:=(H\backslash M) \subset (H\backslash G)=X^\circ\subset X
$$
and
$$
a(L), b(L), \quad \text{ resp. } \quad a(L|_Y), b(L|_Y)
$$
are the geometric invariants defined in Section~\ref{sect:geometry}. 
When $L$ is balanced, Equation~\ref{eqn:bal} follows, by definition.

Let $K\subset G(\A)$ be a compact subset. Consider translates $M_k$ of $M$ by $k\in K$. 
The asymptotic of 
$$
\vol(B_T\cap  \{x_i^{-1} x_j = k\} )
$$ 
is determined by analytic properties of the height integral
$$
I(\mathbf s, k):=\int_{Y^\circ(\A)}\mathsf H(\mathbf s, yk)^{-1} \, \mathrm dy=
\prod_v \int_{Y^\circ(F_v)}\mathsf H_v(\mathbf s, y_vk_v)^{-1} \, \mathrm dy_v
$$
where $Y^\circ=H\backslash M$ and $\mathrm dy$, $\mathrm dy_v$ are 
suitably normalized Haar measures.
Note that the adelic function 
$$
k\mapsto \mathsf H(\mathbf s, yk), 
$$
is continuous, with $\mathsf H_v(\mathbf s, y_vk_v)= \mathsf H_v(\mathbf s, y_v)$
for all but finitely many $v$.
Specialize the integral $I(\mathbf s, k)$
to $\mathbf s = sL$. We know that each local integral 
$$
\int_{Y^\circ(F_v)}\mathsf H_v(sL, y_vk_v)^{-1} \, \mathrm dy_v
$$
is holomorphic for $\Re(s)>a(L|_Y)-\epsilon$, for some $\epsilon >0$, and that
the Euler product $I(sL, k)$ has an isolated pole at $s=a:=a(L|_Y)$ 
of order $b:=b(L|_Y)$. 
When $L$ is balanced, Equation~\ref{eqn:bal} holds for translates $M_k$.

Moreover, the function 
$$
k\mapsto (s-a)^b\cdot I(sL, k)
$$
uniformly continuous and nonvanishing, 
for $\Re(s)>a-\epsilon$, since only finitely many $v$ are affected 
and the local integrals vary uniformly continuously with $k$. 
We conclude that 
$$
s \mapsto \int_{K} I(sL, k) \,\mathrm dk
$$
has an isolated pole at $s=a$ of order $b$.   
It follows that, for  $T\ra \infty$, 
$$
\vol(B_T\cap  \{ x_i^{-1} x_j \in K\})=\int_{K} \vol(B_T\cap  \{ x_i^{-1} x_j=k\})\,
\mathrm dk  = c \,T^a\log(T)^{b-1}(1+o(1)), 
$$
with some constant $c>0$. 
\end{proof}

\begin{rem}
\label{rem:balance}
If the height function is not balanced, 
the proper subvariety defined by 
$$
\{x_i^{-1}x_j = \text{constant}\}
$$ 
contributes a positive proportion of rational points to the asymptotic. 
This is an example of the saturation phenomenon observed in \cite{BT}, 
cf. Remark ~\ref{rem:abc}.
\end{rem}

As a corollary of Theorem \ref{thm:mult} we obtain a result about equidistribution of
period integrals on the space $Z_V=G(F) \backslash G_V$.
We denote by $\mathrm dy$ and $\mathrm dz$ the normalised Haar measures supported
on $Y_W=H(F)\backslash H_W$ and $Z_V=G(F)\backslash G_V$ respectively.
Let $\mathrm dx$ denote the restriction of the Tamagawa measure on $X_V$.
We consider $Y_W$ as a subspace of $Z_V$ embedded in $Z_V$ diagonally.

\begin{coro}\label{c:period}
If the line bundle $L$ is balanced, then for every $f\in \mathsf C_c(Z_V)$,
$$
\lim_{T\ra \infty}   \frac{1}{\vol (B_T)} \int_{B_T} \left(
  \int_{Y_W} f(yx) \, \mathrm dy\right)\, \mathrm dx =  \int_{Z_V} f \, \mathrm dz.
$$
\end{coro}

\begin{proof}
By the Stone-Weierstrass theorem, it suffices to consider functions of the form 
$f = f_1 \otimes \cdots \otimes f_n$ with 
$f_i \in \mathsf C_c(Y_W)$. In this case,
$$
I(x):=\int_{Y_W} f(yx) \, \mathrm dy=\int_{Y_W} f_1(yx_1) \cdots f_r(y x_r) \, \mathrm d y.
$$
Since $B_T$ is invariant under $V=W \times \cdots \times W$,
$$
\int_{B_T} I(x)\, \mathrm dx=
\int_{B_T} \int_{Y_W} \bar f_1(yx_1) \cdots \bar f_r(y x_r) \, \mathrm d y\, \mathrm dx,
$$
where $\bar f_i(y)=\int_W f_i(yw)\, \mathrm dw$,
where $\mathrm dw$ denotes the normalised Haar measure on $W$.
Hence, we may assume that $f_i$'s are $W$-invariant. 

Given a compact subset $K$ of $H_W$, we set  
$$
B_T(K) = \{x \in B_T\, |\,  x_i^{-1} x_j \not\in K, \quad \forall i \ne j\}.
$$ 
By Theorem~\ref{thm:mult}, for every $\epsilon>0$,
there exists a compact subset $K$ of $H_W$ such that for all $x=(x_1, \ldots, x_r) \in B_T(K)$, we have
\begin{equation*}
\left| I(x) - \left(\int_{Y_W} f_1 \, \mathrm dy\right)
\cdots \left(\int_{Y_W} f_r\, \mathrm dy \right) \right| < \epsilon,
\end{equation*}
and 
\begin{equation}\label{eq:b1}
\int_{B_T(K)} I(x) \, \mathrm dx= 
\vol(B_T(K)) \left(\int_{Y_W} f_1 \, \mathrm dy\right)
\cdots \left(\int_{Y_W} f_r\, \mathrm dy \right) +O(\epsilon\, \vol(B_T(K))).
\end{equation}
Also,
\begin{equation}\label{eq:b2}
\int_{B_T\setminus B_T(K)} I(x) \, \mathrm dx=
O(\vol(B_T\setminus B_T(K))).
\end{equation}
Since the line bundle is balanced, it follows from Lemma \ref{lem:balanced} that
$$
\frac{\vol(B_T \setminus B_T(K))}{\vol(B_T)}\to 0\quad\hbox{as $T\to\infty$.}
$$
Hence, combining (\ref{eq:b1}) and  (\ref{eq:b2}), we deduce that
$$
\limsup_{T\to\infty} \left| \frac{1}{\vol(B_T)} \int_{B_T} I(x)
\, \mathrm dx -\left(\int_{Y_W} f_1\,\mathrm dy\right) \cdots
\left(\int_{Y_W} f_r\, \mathrm dy\right)\right| \leq \epsilon
$$
for every $\epsilon>0$, which proves the corollary.
\end{proof}

From Corollary \ref{c:period}, we deduce:
\begin{thm}
If the line bundle $L$ is balanced then
\begin{align*}
\left| X^\circ(F) \cap B_T \right|  &= \vol(H(F)\backslash H_W)^{1-r}\cdot \vol (B_T) (1+o(1))\\
&=  c\,\vol(H(F)\backslash H_W)^{1-r} \cdot T^{a(L)} (\log T)^{b(L)-1} (1+o(1)), \quad\hbox{as $T\to\infty$},
\end{align*}
where $c$ is as in Lemma \ref{lem:volume}.
\end{thm}

\begin{proof}

Let $\mathrm d h$ be the Tamagawa measure on $H(\A)$ restricted to $H_W$.
Then the Haar measure $\mathrm d g$ on $G_V$ can be written as
$$
\int_{G_V} \tilde f\, \mathrm d g=\int_{X_V}\left(\int_{H_W} \tilde f(hx)\,\mathrm d h\right) \mathrm d x,\quad
\tilde f \in \mathsf C_c(G_V).
$$

Let $\tilde f \in \mathsf C_c(G_V)$ be a nonnegative function with
$\hbox{supp}(\tilde f)\subset U$ and $\int_{G_V} \tilde f\, \mathrm dg=1$.
Put
$$
f(g):=\sum_{\gamma\in G(F)} \tilde f(\gamma^{-1} g).
$$
Then, for every $x\in X_V$,
\begin{equation}\label{eq:eq}
\int_{Y_W} f(yx)\, \mathrm d y= \frac{1}{\vol(H(F)\backslash H_W)}
\sum_{\gamma\in H(F)\backslash G(F)} \int_{H_W} \tilde f(\gamma^{-1} hx)\, \mathrm d h.
\end{equation}
If $x\in B_{\kappa^{-1}T}$ and $\gamma^{-1}hx\in U$, then
$\gamma\in hxU$; using (\ref{eq:bt}) we have $H_W \gamma\in B_T$. Hence, (\ref{eq:eq}) implies that
\begin{align}\label{eq:1}
\vol(H(F)\backslash H_W) \int_{B_{\kappa^{-1}T}}&\left(\int_{Y_W} f(yx)\, \mathrm d y\right)\, \mathrm d x\\
\nonumber
=& \sum_{\gamma\in H(F)\backslash G(F)\cap B_T} \int_{H_W\times B_{\kappa^{-1}T}} 
\tilde f(\gamma^{-1} hx)\, \mathrm d h\mathrm d x\\
\nonumber
\le & \sum_{\gamma\in H(F)\backslash G(F)\cap B_T} \int_{G_V} 
\tilde f(\gamma^{-1} g)\, \mathrm d g\\
\nonumber
= &  |H(F)\backslash G(F)\cap B_T|
\end{align}
If $\gamma\in B_T$ and $\gamma^{-1}hx\in U$, then $x\in h^{-1}\gamma U$;
using (\ref{eq:bt}) we have $x\in B_{\kappa T}$.
Now (\ref{eq:eq}) implies that
\begin{align}\label{eq:2}
|H(F)\backslash G(F)\cap B_T| = &
\sum_{\gamma\in H(F)\backslash G(F)\cap B_T} \int_{G_V} 
\tilde f(\gamma^{-1} g)\, \mathrm d g\\
\nonumber = & 
\sum_{\gamma\in H(F)\backslash G(F)\cap B_T} \int_{H_W\times X_V} 
\tilde f(\gamma^{-1} hx)\, \mathrm d h\mathrm d x\\
\nonumber \le & 
\sum_{\gamma\in H(F)\backslash G(F)} \int_{H_W\times B_{\kappa T}} 
\tilde f(\gamma^{-1} hx)\, \mathrm d h\mathrm d x\\
\nonumber = & 
\,\vol(H(F)\backslash H_W)\int_{B_{\kappa T}}\left(\int_{Y_W} f(yx)\, \mathrm d y\right)\, \mathrm d x.
\end{align}

By Lemma \ref{lem:volume},
$$
\lim_{T\to\infty}
\frac{\vol(B_{\kappa T})}{\vol(B_T)}=\kappa^{a(L)}.
$$
Combining (\ref{eq:1}) with Corollary \ref{c:period}, we obtain
\begin{align*}
&\liminf_{T\to\infty}
\frac{|H(F)\backslash G(F)\cap B_T|}{\vol(B_T)}
\\
\ge& \left(\lim_{T\to\infty}
\frac{\vol(B_{\kappa^{-1}T})}{\vol(B_T)}\right)
\left(\lim_{T\to\infty} \frac{\vol(H(F)\backslash H_W)}{\vol(B_{\kappa^{-1}T})}
\int_{B_{\kappa^{-1}T}}\left(\int_{Y_W} f(yx)\, \mathrm d y\right)\, \mathrm d x\right)\\
=& \kappa^{-a(L)} \vol(H(F)\backslash H_W) \int_{Z_V} f\, \mathrm d z\\
=& \kappa^{-a(L)}\, \frac{\vol(H(F)\backslash H_W)}{\vol(G(F)\backslash G_V)}
=\kappa^{-a(L)}\, \vol(H(F)\backslash H_W)^{1-r}.
\end{align*}
Similarly, it follows from (\ref{eq:2}) that
\begin{align*}
\limsup_{T\to\infty}
\frac{|H(F)\backslash G(F)\cap B_T|}{\vol(B_T)}
\le \kappa^{ a(L)} \vol(H(F)\backslash H_W)^{1-r}.
\end{align*}
Since these estimates hold for all $\kappa>1$, we conclude that
$$
|H(F)\backslash G(F)\cap B_T|=
\vol(H(F)\backslash H_W)^{1-r} \vol(B_T)(1+o(1))
$$
as $T\to\infty$. 
Since $X^\circ(F)=H(F)\backslash G(F)$ by Corollary \ref{coro:gn},
this proves the first part of the theorem.
The second part follows from Lemma \ref{lem:volume}.
\end{proof}

Theorem~\ref{thm:main} follows by applying Proposition~\ref{prop:subgroup}, which insures
that the anticanonical line bundle $-K_X$ is balanced.


\end{document}